\documentclass[a4paper,12pt, reqno]{amsart}
\usepackage{amssymb,amsthm,amsmath}
\usepackage{cite}

\usepackage{amsmath,amsthm,amscd,amssymb}
\usepackage{latexsym}
\usepackage[colorlinks,citecolor=red,pagebackref,hypertexnames=false]{hyperref}
 
\numberwithin{equation}{section}
\allowdisplaybreaks[2]
\theoremstyle{plain}
\newtheorem{theorem}{Theorem}[section]

\theoremstyle{definition}

\theoremstyle{remark}
\newtheorem{remark}[theorem]{Remark}

\newtheorem{case[theorem]}{Case}

\def\norm#1.#2.{\lVert#1\rVert_{#2}}

\title[A note on Singular integral]{A note on Singular integral
}

\author{Arup Maity } \address{Arup Maity \endgraf School of Mathematics,	\endgraf Harish-Chandra Research Institute, 	\endgraf Allahabad, 211 019, India} \email{ arupkumarmaity@hri.res.in}

\author{Shyam Swarup Mondal } \address{Shyam Swarup Mondal  \endgraf Department of Mathematics	\endgraf Indian Institute of Technology Guwahati,	\endgraf Guwahati-781039, Assam, India.} \email{mondalshyam055@gmail.com}



\date{\today}
\begin{document}
\maketitle
\begin{abstract} 
In this paper, we prove  that for $\frac{n}{2}+\frac{1}{4}<\alpha \leq\frac{n+1}{2} $, the convolution operator    
$$S_{\alpha} f(x)=\int_{|y| \geq  1} f(x-y)\left(|y|^{2}-1\right)^{-\alpha} d y$$
is  bounded from $L^p$ to $L^q$ for certain values of $p$ and $q$.  
\end{abstract} 
\section{Introduction}
In this paper, we consider the $L^p-L^q$ estimates for the convolution operators having  kernel  with singularity along the sphere and at infinity which is non zero outside the neighbourhood of zero.

 For convolutions on Euclidean space with homogeneous kernels, the fractional integration theorem and the Calderón-Zygmund inequalities give essentially the best possible results \cite{hor}.  Such  kind of kernels have singularities at the point $0$ and $\infty$. It was Strichartz who first proved that other than that kind of singularity,  the corresponding operator is still bounded from $L^p$ to other $L^q$.  Such kind of operators  naturally arise in the study of the wave equation and   gives   new information about solutions of this type of equation.

 For $0<\alpha<1$,   Strichartz considered the operator
 $$T_{\alpha} f(x)=\int_{|y| \leq   1} f(x-y)\phi_{\alpha}(y)\;d y,$$
  where the  kernel is defined by 
\begin{align}\label{ker}
\phi_{\alpha}(y)=\begin{cases}
	(1-|y|^2)^{-\alpha},& \text{if}\; |y|<1,\\
	0,&   \text{if}\; |y|\geq 1.
\end{cases}	
\end{align}
 The author investigated  $L^p-L^q$ boundedness of such convolution operator with kernel  having  singularities on the sphere and at infinity.  In general, from \cite{parker13}, if we consider  the   family of generalized functions $(ax^2+bx+c)_+^z$,  where 
 
 $$(ax^2+bx+c)_+^z=\begin{cases}
 	(ax^2+bx+c)^z,& ~\text{ if  }  \quad ax^2+bx+c>0,\\
 	0,& ~\text{ if  } \quad ax^2+bx+c\leq0,
 \end{cases} $$
then it is always possible to perform a linear transformation on $x$ such that 
the  generalized function $(ax^2+bx+c)_+^z$ is transformed to one of the   forms   $\left(1-x^{2}\right)_+^\lambda$, $\left(1+x^{2}\right)_+^\lambda$, $\left(x^{2}-1\right)_+^\lambda$, or  $x_+^{2\lambda}$ respectively, to the four cases   in which $a x^{2}+b x+c>0$.  Here in this paper, we mainly interested in the case $\left(x^{2}-1\right)_+^\lambda$.  Note that, all other cases  have becn treated already. In particular, Strichartz in \cite{hopkins2} cosider the first case $\left(1-x^{2}\right)_+^\lambda$ and  proved the following  result.
 \begin{theorem}\cite{hopkins2}
 	\begin{enumerate}
 		\item  $\left\|T_{\alpha} f\right\|_{q} \leq A_{\alpha}\|f\|_{p}$ provided $1<p \leq 2 \leq q<\infty  $ and $ \frac{1}{p}-\frac{1}{q}  \leq\frac{(n+1-2 \alpha)}{2n} $ for $0<\alpha \leq\frac{(n+1)}{2}$.\\
 		\item $\left\|T_{\alpha} f\right\|_{p^{\prime}} \leq A_{\alpha}\|f\|_{p}$ for $p=\frac{(n+1)}{(n+1-\alpha)}, p^{\prime}=\frac{(n+1)}{\alpha}$ and $0<\alpha \leq \frac{(n+1)} {2}$.\\
 		\item $\left\|T_{\alpha} f\right\|_{q} \leq A_{\alpha}\|f\|_{p}$ for $\frac{1}{2} \leq \alpha \leq\frac{(n+1)} {2}$ provided $\frac{(n+1)} {(n+1-\alpha)} \leq p \leq 2$ and $\frac{n}{q}=\alpha-\frac{1}{p}$, or $\frac{n}{\left(n+\frac{1}{2}-\alpha\right)} \leq p \leq\frac{(n+1)} {(n+1-\alpha)}$ and $\frac{1}{q}=\alpha-\frac{n}{p^{\prime} }.$
 	\end{enumerate}
 \end{theorem}
Considerable attention has been devoted to discovering generalizations to new contexts for the boundedness of convolution operator by several researchers in different contexts.  For example,  $L^p-L^q$  boundedness of convolution operator was studied  by  Amri  and  Gaidi  in \cite{gai}  in the Dunkel setting and  investigated  the solutions of wave equations associated to Dunkl Laplacian operator. Karapetyants studied  $L^p-L^q$ boundedness for convolutions operator  with  
kernels having singularities on a sphere in \cite{alex}.  F. Ricci   and T. Giancarlo in \cite{ricci} studied   boundedness of  convolution operators defined by singular measures.   Further, $\mathcal{L}$-characteristic of some  potential type  convolution operator $A^{\alpha}$ is bounded from $L^{p}$ into $L^{q}$ was  given in the works of  Karapetyants and  Nogin \cite{kara}.   There are several papers in the direction  of bounds for    potential-type convolution operators and close to them, we refer to  
 \cite{1,2,3,4,5, gil} and references therein.    Such operators naturally arise in applications to the theory of fractional powers of differential operators, in particular, classical operators of mathematical physics: the wave operator, the Klein-Gordon and Schrödinger operators, the telegraph operator.



 
 The main aim of this paper is to prove  $L^p$ to $L^q$ bounded of the   convolution operator    
 \begin{align}\label{convolution}
 S_{\alpha} f(x)=\int_{|y| \geq  1} f(x-y)\xi_{\alpha}(y)\; d y
 \end{align}
  for $\alpha>0$, where the kernel is defined by 
\begin{align}\label{ker1}
	\xi_{\alpha}(y)=\begin{cases}
		(|y|^2-1)^{-\alpha},& \text{if}\; |y|>1,\\
0,&  \text{if}\; |y|\leq1,
\end{cases}
\end{align} 
  for certain $p$ and $q$.   The  kernel defined in (\ref{ker1}) is different form the kernel  (\ref{ker}), considered by Strichartz. The main difference  in the kernel defined in   (\ref{ker1}) is that this  kernel has singularity along the sphere and in infinity.  Also it is non zero outside the neighbourhood of zero which is not consider by any author in past and  which makes the problem rather slightly different. However,  in the sense of distribution, the range of $\alpha$ can be extended similarly as in Equation (3) of  \cite{hopkins2} by  integration by parts.

  Our main resuls in this article  is the following $L^p-L^q$ boundedness theorem of convolution operator $S_\alpha$.
  \begin{theorem}\label{Main}
  	Let $\frac{n}{2}+\frac{1}{4}<\alpha \leq\frac{n+1}{2} $.  Then  for $1<p\leq 2\leq q<\infty$ and $$\frac{1}{p}-\frac{1}{q} \leq  \frac{2\alpha-n-\frac{1}{2}}{n},$$ 
  	the   convolution opeartor $$S_{\alpha}f=\xi_{\alpha}*f,\quad f \in \mathcal{S}(\mathbb{R}^n)$$
  	 is bounded from $L^p$ to $L^q$.
  \end{theorem}
 The proof of the theorem is given in Section \ref{sec2}.  The main ingridient  we use is the the Hardy-Littlewood multiplier theorem given in Theorem 1.11 of   \cite{hor},  to prove Theorem \ref{Main} using  a similar technique  given in   \cite{hopkins2}.

The presentation of this manuscript is divided into three sections including the introduction. In section \ref{sec3}, we recall some basics properties and results from eucledian  Fourier analysis which will be used in order to prove our main results. In Section \ref{sec2},  we prove our main reults.

 \section{Preliminaries}\label{sec3}
In this section,  we first recall some notation and basic properties of   Fourier
analysis    on $\mathbb{R}^n$. 
The  Fourier transform $\hat{f}$ of a function $f \in L^{1}\left(\mathbb{R}^{n}\right)$ is defined by
$$
\mathcal{F}f(\xi)=\widehat{f}(\xi)=\int_ {\mathbb{R}^{n}} f(x) e^{i x \cdot \xi} d x, \quad \xi \in \mathbb{R}^{n}.
$$
The   Fourier transform can be extended to $L^{2}\left(\mathbb{R}^{n}\right)$ using the standard density arguments. 
The inverse  Fourier transform is given by
$$
f(x)=\int_{\mathbb{R}^{n}} \hat{f}(\xi) e^{-i x \cdot \xi} d \xi, \quad x \in \mathbb{R}^{n},
$$ where $f$ belongs to a suitable function space. The space  $  S\left(\mathbb{R}^{n}\right)$ is the Schwartz class of rapidly decreasing smooth functions and $S^{\prime}\left(\mathbb{R}^{n}\right)$ is the space of tempered distributions. Let $T\in S^{\prime}\left(\mathbb{R}^{n}\right)$, its fourier transform $\hat{T}$ defined in the sense of distribution by
$$\widehat{T}(\phi)=T(\hat{\phi}), \quad \phi \in S(\mathbb{R}^n).$$
Moreover, $\widehat{T}$ is called multiplier if it is $L^p-L^q$ bounded.

As usual, the space $L_{p}^{q}$ denotes the Banach space of $S^{\prime}(\mathbb{R}^n)$-distributions $T$ such that the closure in $L^{p}$ of the convolution $T \ast u$ for  $u \in S(\mathbb{R}^n)$, is an $L^{p} -L^{q}$ bounded translation invariant operator. The spcae $L_{p}{ }^{q}$ is thus isomorphic to a closed subspace of the Banach space of all bounded linear mappings of $L^{p}$ into $L^{q}$ and  hence is also a Banach space. Also,  $M_{p}^{q}=\mathcal{F} L_{p}^{q}$ stands for the space of  all $L^{p}-L^{q}$ multipliers  type $(p, q)$ \cite{hor}. Moreover, we have
$$
L_{p}^{\infty}=L_{1}^{p^{\prime}}=L^{p^{\prime}}, p<\infty~\text{and}~ L_{\infty}^{\infty}=L_{1}^{1}=M.
$$
 From \cite{hor}, we have the following Hardy-Littlewood multiplier theorem. 
\begin{theorem}\cite{hor}\label{multipliers}
	 Let $0<a<n$ and $m$ be a measurable function such that  
	$$
	|m(\xi)| \leq \frac{c}{|\xi|^{a}}
	$$
	some constant $c>0$. 	Then the operator $T_{m}=\mathcal{F}^{-1}\left(m \mathcal{F} \right)$ is bounded from $L^{p}\left(\mathbb{R}^{n} \right)$ to $L^{q}\left(\mathbb{R}^{n} \right)$, provided
	$$
	1<p \leq 2 \leq q<\infty ~~\text{and}~~\frac{1}{p}-\frac{1}{q}=\frac{a}{n} .
	$$
\end{theorem}

By $J_{\left( \alpha\right)}(s)$, we  denotes  the classical  Bessel function of size $\alpha$. From \cite{thirteen}, we have the following estimate:
 \begin{align}\label{1001}
 \left|t^{-(a+i b)} J_{\left( a+ib\right)}(t)\right| \leq C_{a}~ e^{c|b|}~(1+t)^{-a-\frac{1}{2}}, \quad 0<t<\infty.
 \end{align}
 For detailed study and several properties on    Bessel functions, we refer to   \cite{hopkins3, thirteen}. 

\section{$L^p$-$L^q$-boundedness}\label{sec2}
This section is devoted to study our main results of this paper. The Hardy-Littlewood multiplier theorem plays a  key role in the proof of the theorem.
\subsection{Proof of Theorem \ref{Main}}

\begin{proof}
	For $\alpha>0$, let us consider $$\xi_{\alpha}(y)=\begin{cases}
			(|y|^2-1)^{-\alpha},& \text{if}\; |y|>1,\\
		0,&   \text{if}\; |y|\leq1.
	\end{cases}$$
Then the Fourier transform of $\xi_\alpha$ is not classically defined, as the function $\xi_\alpha$ is not in $L^1(\mathbb{R}^n)$.   We define it by the sense of  tempered distribution.  For $\alpha \neq n, n \in \mathbb{N}$, we define
	$$\widehat{\xi_{\alpha}}(s)=\sqrt{\pi} \;\Gamma(1-\alpha) \frac{|s|}{2}^{\left(\alpha-\frac{n}{2}\right)}\cot\pi\left(\frac{n}{2}-\alpha\right)\left(J_{\left(\alpha-\frac{n}{2}\right)}(|s|)-J_{\left(\frac{n}{2}-\alpha\right)}(|s|)\right)$$
	for $0\leq \alpha<1$, where $J_{\left( \alpha\right)}(s)$ denotes  the classical  Bessel function of size $\alpha$. This is well defined by the analytic continuation of $\Gamma$ almost everywhere. Since  
\begin{align}\label{eq5}
S_{\alpha}f=\xi_{\alpha}*f, \quad f \in \mathcal{S}(\mathbb{R}^n),
\end{align}
	then the Fourier transform of $S_{\alpha}f$ is given by 
	\begin{align}\label{eq1}
	\widehat{S_{\alpha}f}(s)=\sqrt{\pi} \;\Gamma(1-\alpha) \frac{|s|}{2}^{\left(\alpha-\frac{n}{2}\right)}\cot\pi\left(\frac{n}{2}-\alpha\right)\left(J_{\left(\alpha-\frac{n}{2}\right)}(|s|)-J_{\left(\frac{n}{2}-\alpha\right)}(|s|)\right)\hat{f}(s).
		\end{align}
The right-hand side of the above inequality  is a single-valued analytic function of $\alpha$ in the complex plane except at  the positive integers. 

	Let $\frac{n}{2}+\frac{1}{4}<\alpha \leq\frac{n+1}{2} $ and     consider the function
	\begin{align}\label{eq2}
		m(s)=\sqrt{\pi} \;\Gamma(1-\alpha) \;\frac{|s|}{2}^{\left(\alpha-\frac{n}{2}\right)}\cot\pi\left(\frac{n}{2}-\alpha\right)\left(J_{\left(\alpha-\frac{n}{2}\right)}(|s|)-J_{\left(\frac{n}{2}-\alpha\right)}(|s|)\right).
	\end{align}
Using the classical estimate (\ref{1001}) on the size of Bessel functions, we have   
\begin{align}\label{eq3} J_{\left(\alpha-\frac{n}{2}\right)}(|s|) \leq (1+|s|)^{-\alpha +\frac{n}{2}-\frac{1}{2}}	\end{align}
and 
\begin{align}\label{eq4} J_{\left(\frac{n}{2}-\alpha\right)}(|s|) \leq (1+|s|)^{\alpha -\frac{n}{2}-\frac{1}{2}}.	\end{align}
Using the estimates (\ref{eq2}) and (\ref{eq3}), from  the relation  (\ref{eq4}),  we get 
$$|m(s)| \leq c|s|^{-\left( 2\alpha-n-\frac{1}{2}\right)} ,$$
for some costant $c$.  Thus by Theorem \ref{multipliers},    $m$ is a $L^p-L^q$ multiplier provided 
$$1 \leq p \leq 2 \leq q < \infty \quad \text{and}\quad  ~\frac{1}{p}-\frac{1}{q}= \frac{2\alpha-n-\frac{1}{2}}{n}.$$ 	Thus,  the operator $S_\alpha$ is bounded from $L^p$ to $L^q$ for $\frac{1}{p}-\frac{1}{q} \leq  \frac{2\alpha-n-\frac{1}{2}}{n}$. This completes the proof of the theorem. 
\end{proof}

 In the next theorem we show that the convolution operator  $S_{\alpha}$
 is bounded from $L^p$ to $L^{p'}$ in one dimentional. 
 \begin{theorem}\label{main1}
	Let  $0\leq \alpha \leq 1$. Then for $p=\frac{2}{(2-\alpha)} $ and $p^{\prime}=\frac{2}{\alpha}$,  	the   convolution opeartor $S_{\alpha}$
is bounded from $L^p$ to $L^{p'}$.
 \end{theorem}
 \begin{proof}
 Let $S_{\alpha}$ is an family of operators defined in the strip $0 \leq \operatorname{Re}(\alpha) \leq 1$   in the complex plane in the sense of Stein \cite{stein}. 	 When  $\operatorname{Re}(\alpha)=1$, using Plancherel theorem,  (\ref{eq5}), (\ref{eq4}), and (\ref{eq4}), we have 
 	 $$\left\|S_{\alpha} f\right\|_{2} \leq c_1 e^{c \operatorname{Im}(\alpha)}\;\|f\|_{2}.$$
 	  Again on the other side,  when  $\operatorname{Re}(\alpha)=0$, using the fact that $\left|\xi_{\alpha}(y)\right| \leq 1$ and $|\Gamma(1+i b)|=\left(\frac{\pi b}{\sinh b}\right)^{\frac{1}{2}}$, from the relation (\ref{eq1}), we have the estimate 
 	  $$\left\|S_{\alpha} f\right\|_{\infty} \leq c_1 e^{c \operatorname{Im}(\alpha)} \; \|f\|_{1}.$$ 
Then by  Stein's interpolation theorem theorem \cite{stein}, we have the result. 
 \end{proof}
An immediate consequence of the above theorem is the follwing $L^p-L^q$ boundedness of the  convolution opeartor $S_{\alpha} $.
\begin{theorem}
	Let  $\frac{1}{2} \leq \alpha \leq 1$. Then for   $$\frac{2} {(2-\alpha)} \leq p \leq 2\quad \text{and}\quad\frac{1}{q}=\alpha-\frac{1}{p}$$
	 or 
	 $$\frac{1}{\left(\frac{3}{2}-\alpha\right)} \leq p \leq\frac{2} {(2-\alpha)}\quad\text{and}\quad \frac{1}{q}=\alpha-\frac{1}{p^{\prime} },$$
	 	the   convolution opeartor $S_{\alpha} $
	 is bounded from $L^p$ to $L^q$.
\end{theorem}
 \begin{proof}
 If $\alpha>\frac{1}{2},$	then the result follows  from Theorem \ref{Main} and \ref{main1}   by an application of  Riesz interpolation theorem. Now let us consider  $\alpha=\frac{1}{2}$. Since $\xi_{\frac{1}{2}}$ is in weak $L^{2}$,  then $T_{\frac{1}{2}}$ maps $L^{1}$ to weak $L^{2}$. Then the required result follows by applying the Marcinkiewicz interpolation theorem \cite{zyg} and a duality argument \cite{hor}.
 \end{proof}
 
 \begin{remark}
 	Note that,    $\widehat{\xi_\alpha}$  (defined in (\ref{ker1})) is  a multiplier in terms of Bessel functions. 
 \end{remark}

\begin{remark}
If we define kernels by replacing $|y|^{2}$ by any nondegenerate quadratic form   using the computations of [\cite{parker13}, Chapter $\mathrm{III}]$, then all the results are also in this case. 
\end{remark}

\section*{Acknowledgments}
The first author gratefully acknowledges the support provided by  Harish-Chandra Research Institute, Government of India.   Second author thanks IIT Guwahati, India, for the support provided during the period of this work. The  authors  are  deeply indebted to Prof. P. K. Ratnakumar for several fruitful discussions and generous comments.


\end{document}